\documentclass[a4paper,reqno]{amsart}
\usepackage{amssymb,amsthm,amscd,amsmath}
\usepackage[all]{xy}

\newtheorem{theor}{Theorem}

\newtheorem{lemma}[theor]{Lemma}
\newtheorem{prop}[theor]{Proposition}
\theoremstyle{definition}
\newtheorem{defn}{Definition}
\newtheorem{rem}{Remark}

\allowdisplaybreaks[4]

\newcommand*{\abs}[1]{\lvert#1\rvert}

\newcommand*{\pd}[2]{\frac{\partial#1}{\partial#2}}

\def\cprime{\/{\mathsurround=0pt$'$}}

\newcommand{\ldb}{[\![}
\newcommand{\rdb}{]\!]}
\DeclareMathOperator{\sym}{sym} \DeclareMathOperator{\Hom}{Hom}
\DeclareMathOperator{\CDiff}{\mathcal{C}Diff}

\newcommand{\CDiffskad}{\CDiff^{\,\text{\textup{sk-ad}}}}
\newcommand{\CDiffskew}{\CDiff^{\,\text{\textup{skew}}}}
\renewcommand{\kappa}{\varkappa}
\renewcommand{\phi}{\varphi}

\DeclareFontFamily{OML}{cyi}{} \DeclareFontShape{OML}{cyi}{m}{n}{
  <5> <6> <7> <8> <9> gen * wncyi
  <10> <10.95> <12> <14.4> <17.28> <20.74> <24.88> wncyi10
 }{}
\DeclareSymbolFont{rusletters}{OML}{cyi}{m}{n}
\DeclareSymbolFontAlphabet{\rusmath}{rusletters}
\DeclareMathSymbol\re{\rusmath}{rusletters}{"03}

\begin{document}

\title{Variational Poisson--Nijenhuis structures for partial differential
  equations} %\date{June, 2006}

\author{V.~Golovko}
\address{Valentina Golovko\\
  Lomonosov MSU\\
  Faculty of Physics, Department of Mathematics,
  Vorob'evy Hills, Moscow 119902 Russia.}
\email{golovko@mccme.ru}

\author{I.~Krasil{\cprime}shchik}
\address{Iosif Krasil{\cprime}shchik \\
  Independent University of Moscow \\
  B. Vlasevsky~11 \\
  119002 Moscow \\
  Russia}

\email{josephk@diffiety.ac.ru}

\author{A.~Verbovetsky}
\address{Alexander Verbovetsky \\
  Independent University of Moscow \\
  B. Vlasevsky~11 \\
  119002 Moscow \\
  Russia}

\email{verbovet@mccme.ru}



\thanks{This work was supported in part by the NWO--RFBR grant 047.017.015 and
  RFBR--Consortium E.I.N.S.T.E.I.N grant 06-01-92060.}

\begin{abstract}
  We explore variational Poisson--Nijenhuis structures on nonlinear PDEs and
  establish relations between Schouten and Nijenhuis brackets on the initial
  equation with the Lie bracket of symmetries on its natural extensions
  (coverings). This approach allows to construct a framework for the theory of
  nonlocal structures.
\end{abstract}

\maketitle

% ------------------------------------------------------------------------------

\section*{Introduction}
\label{sec:introduction}

Poisson--Nijenhuis structures~\cite{V2} play an important role both in
classical differential geometry (see, for example~\cite{B-M, KS}) and in
geometry of partial differential equations, see~\cite{KS-Ma, Ma-Mo}. In the
latter case existence of a Poisson--Nijenhuis structure virtually amounts to
complete integrability of the equation under consideration.

Infinite-dimensional Poisson--Nijenhuis structures are well described for the
case of jets and for evolutionary differential equations regarded as flows on
the jet space. As for general differential equation, the corresponding theory
was not introduced for a long time. In our relatively recent works~\cite{Ham,
  Bussinesq} we outlined an approach to the theory in application to evolution
equation in geometrical setting. This approach is based on the notion of
$\Delta$-coverings and reduces construction of both recursion operators and
Hamiltonian structures to solution of the linearised equation
\begin{equation}\label{eq:23}
  \ell_{\mathcal{E}}(\Phi)=0
\end{equation}
on special extensions of the initial equation~$\mathcal{E}$. We call these
extensions the \emph{$\ell$-} and \emph{$\ell^*$-coverings} and they play the
role of tangent and cotangent bundles in the category of differential
equations.

The above mentioned approach seems to work for general equations as well and
we expose its generalisation below.

In Section~\ref{sec:preliminaries} we define Poisson--Nijenhuis structures in
the ``absolute case'', i.e., for the manifold of infinite jets. To this end,
we redefine the Schouten and Fr\"{o}licher--Nijenhuis brackets
(cf.~with~\cite{Ham, C-cohom}, see also~\cite{Vin-book}). We also express the
compatibility condition between a Poisson bi-vector and a Nijenhuis operator
in terms of a special bracket closely related to Vinogradov's \emph{unified
  bracket}, see~\cite{Vin}. The main result of this section is
Theorem~\ref{theor:main} that states the existence of infinite families of
pair-wise compatible Hamiltonian structures related to the initial
Poisson--Nijenhuis structure. In Section~\ref{sec:vari-poiss-nijenh}
Poisson--Nijenhuis structures on evolution equations are introduced. We show
that an invariant (with respect to the flow determined by the equation)
Nijenhuis tensors are \emph{recursion operators} for the symmetries, while
invariant Poisson bi-vectors amount to Hamiltonian structures. We also define
the $\ell$- and $\ell^*$-coverings and reduce construction of recursion
operators and Hamiltonian structures to solution of
equation~\eqref{eq:23}. The Schouten and Fr\"{o}licher--Nijenhuis brackets as
well as the compatibility conditions are reformulated in terms of the Jacobi
brackets of the corresponding solutions and explicit formulas for these
brackets are obtained. Section~\ref{sec:gen-str} generalises the results to
arbitrary nonlinear partial differential equation. Finally, in
Section~\ref{sec:nonloc-str} we outline an approach to deal with
\emph{nonlocal} Poisson--Nijenhuis structures.

%---------------------------------------------------------------------------------------------

\section{Variational Poisson--Nijenhuis structures on $J^\infty(\pi)$}
\label{sec:preliminaries}

\subsection{Geometrical structures}

Let us recall definitions and results we shall use. For details we refer
to~\cite{ClassSymEng}.

Let $\pi\colon E\rightarrow M$ be a vector bundle over an $n$-dimensional
manifold $M$ and $\pi_\infty\colon J^\infty(\pi)\rightarrow M$ be the infinite
jet bundle of local sections of the bundle~$\pi$. If $x_1,\dots,x_n$ are local
coordinates in the base and $u^1,\dots,u^m$ are coordinates along the fiber
of~$\pi$ the \emph{canonical coordinates}~$u_\sigma^j$ arise
in~$J^\infty(\pi)$ defined by
\begin{equation*}
  j_\infty(s)^*(u_\sigma^j)=\frac{\partial^{|\sigma|}s^j}{\partial x_\sigma},
\end{equation*}
where~$s=(s^1,\dots,s^m)$ is a local section of $\pi$, $j_\infty(s)$ is its
infinite jet and~$\sigma=i_1i_2\dots i_{\abs{\sigma}}$, $i_\alpha=1,\dots,n$,
is a multi-index. Denote by $\mathcal{F}(\pi)$ the algebra of smooth functions on
$J^\infty(\pi)$.

The basic geometrical structure on $J^\infty(\pi)$ is the \emph{Cartan
  distribution} $\mathcal{C}$ that is spanned by \emph{total derivatives}
\begin{equation*}
  D_i=\frac{\partial}{\partial x_i}+\sum_{j,\sigma} u^j_{\sigma i}\frac{\partial}{\partial u^j_\sigma}.
\end{equation*}
Differential operators on $ J^\infty(\pi)$ in total derivatives will be called
$\mathcal{C}$-\emph{differential operators}. In local coordinates, they have the form
$\|\sum_\sigma a^\sigma_{ij}D_\sigma\|$, where $a^\sigma_{ij}\in\mathcal{F}(\pi)$.
Let $P$ and $Q$ be $\mathcal{F}(\pi)$-modules of sections of some vector bundles over
$J^\infty(\pi)$. All $\mathcal{C}$-differential operators from a $P$ to $Q$ form an
$\mathcal{F}(\pi)$-module denoted by $\CDiff(P,Q)$.  The \emph{adjoint operator} to a
$\mathcal{C}$-differential operator $\Delta\colon P\rightarrow Q$ is denoted by
$\Delta^*\colon \hat{Q} \rightarrow\hat{P}$, where
$\hat{P}=\Hom_{\mathcal{F}(\pi)}(P,\bar\Lambda^n(\pi))$ and $\bar\Lambda^n(\pi)$ is
the $\mathcal{F}(\pi)$-module of \emph{horizontal} $n$-form on $J^\infty(\pi)$, i.e.,
forms $\omega=a\,dx_1\wedge\dots\wedge\,dx_n$.

Denote by $\CDiffskew_{(k)}(P,Q)$ the module of $k$-linear skew-symmetric
$\mathcal{C}$-differential operators $P\times\dots\times P\rightarrow Q$ and by
$\CDiffskad_{(k)}(P,\hat{P})\subset\CDiffskew_{(k)}(P,\hat P)$ the subset of
operators skew-adjoint in each argument.

A $\pi_\infty$-vertical vector field on $J^\infty(\pi)$ is called
\emph{evolutionary} if it preserves the Cartan distribution. There is a
one-to-one correspondence between evolutionary vector fields and sections of
the bundle~$\pi_\infty^*(\pi)$. Denote by $\kappa(\pi)$ the corresponding
module of sections. In local coordinates, the evolutionary vector field that
corresponds to a section (the \emph{generating} section, or function)
$\phi=(\phi^1,\dots,\phi^m)$ is of the form
\begin{equation*}
  \re_\varphi=\sum_{j,\sigma}D_\sigma(\varphi^j)\frac{\partial}{\partial u^j_\sigma}.
\end{equation*}
The commutator of evolutionary vector fields induces in~$\kappa(\pi)$ a Lie
algebra structure that is given by the \emph{higher Jacobi bracket} defined as
a unique section $\{\varphi,\psi\}$
satisfying~$[\re_\phi,\re_\psi]=\re_{\{\varphi,\,\psi\}}$. The bracket
${\{\varphi,\psi\}}$ is expressed by
\begin{equation*}
  \{\varphi,\psi\}=\re_\varphi(\psi)-\re_\psi(\varphi).
\end{equation*}

Following~\cite{Ham,Mult-vect}, we shall call elements of $\kappa(\pi)$
\emph{variational vectors}, elements of the module
$\CDiffskad_{(k-1)}(\hat\kappa,\kappa)$ will be called \emph{variational
  $k$-vectors}, while elements of $\hat\kappa$ will be called
\emph{variational $1$-forms} and elements of the module
$\CDiffskad_{(k-1)}(\kappa,\hat\kappa)$ variational $k$-\emph{forms},
respectively. The Lie derivative on variational vectors
$L_\phi\colon\kappa\rightarrow\kappa$ takes the form
\begin{equation}
 L_\phi=\re_\phi-\ell_\phi,\label{eq:1}
\end{equation}
where the \emph{linearization operator} $\ell_\phi$ is defined by the
equality $\ell_\phi(\alpha)=\re_\alpha(\phi)$, $\alpha\in\kappa$.  Un
local coordinates, it has the form
\begin{equation*}
  \ell_\phi(\alpha)=\sum_{j,\sigma}\frac{\partial\phi}{\partial u^j_\sigma}D_\sigma(\alpha^j),
  \quad\alpha=(\alpha^1,\dots,\alpha^m).
\end{equation*}
The Lie derivative on variational forms $L_\phi\colon
\hat\kappa\rightarrow\hat\kappa$ is of the form
\begin{equation}\label{eq:3}
  L_\phi=\re_\phi+\ell^*_\phi.
\end{equation}

% ------------------------------------------------------------------

\subsection{Variational Poisson--Nijenhuis structures}

Recall (see~\cite{Ham}) that the \emph{variational Schouten bracket} of two
operators~$A$, $B\in\CDiffskad(\hat\kappa,\kappa)$ is defined by
\begin{multline}\label{var-Scho2}
  \ldb A,B\rdb(\psi_1,\psi_2)=-\ell_{A,\,\psi_1}(B\psi_2)+
  \ell_{A,\,\psi_2}(B\psi_1)-A(\ell^*_{B,\,\psi_1}(\psi_2))\\
  -\ell_{B,\,\psi_1}(A\psi_2)+
  \ell_{B,\,\psi_2}(A\psi_1)-B(\ell^*_{A,\,\psi_1}(\psi_2)),\quad
  \psi_1,\psi_2\in\hat\kappa.
\end{multline}
An operator $A$ is called \emph{Hamiltonian} if the $\ldb A,A\rdb=0$ and two
Hamiltonian operators~$A$ and~$B$ are \emph{compatible} if their Schouten
bracket vanishes.

\begin{rem}
  Here and below the notation~$\ell_{\Delta,p_1,\dots,p_n}(\phi)$,
  $\phi\in\kappa$, for a $\mathcal{C}$-differential operator~$\Delta\colon
  P\times\dots\times P\to Q$ means
  \begin{equation*}
    \ell_{\Delta,p_1,\dots,p_n}(\phi)=\re_\phi(\Delta)(p_1,\dots,p_n),\qquad
    p_1,\dots,p_n\in P.
  \end{equation*}
\end{rem}

For two operators $R$, $S\in\CDiff(\kappa,\kappa)$ their
\emph{Fr\"{o}licher--Nijenhuis bracket} (cf.~with \cite{C-cohom}, see
also~\cite{Vin-book}) is defined by
\begin{multline}\label{FN1}
  [R,S]_{FN}(\phi_1,\phi_2)=\{R\phi_1,S\phi_2\}+\{S\phi_1,R\phi_2\}
  -R(\{S\phi_1,\phi_2\}+\{\phi_1,S\phi_2\} \\
  -S\{\phi_1,\phi_2\})-S(\{R\phi_1,\phi_2\}+
  \{\phi_1,R\phi_2\}-R\{\phi_1,\phi_2\}),\quad\phi_1,\phi_2\in\kappa.
\end{multline}
If $[R,R]_{FN}=0$ we shall refer to $R$ as a \emph{Nijenhuis operator}. For
particular computations it is convenient to use the equality
\begin{multline}
  [R,S]_{FN}(\phi_1,\phi_2)=
  -\ell_{R,\,\phi_1}(S\phi_2)-\ell_{S,\,\phi_1}(R\phi_2)
  +\ell_{R,\,\phi_2}(S\phi_1)+\ell_{S,\,\phi_2}(R\phi_1)\\
  +R\left(\ell_{S,\,\phi_1}(\phi_2)-\ell_{S,\,\phi_2}(\phi_1)\right)
  +S\left(\ell_{R,\,\phi_1}(\phi_2)-\ell_{R,\,\phi_2}(\phi_1)\right).
\end{multline}

\begin{defn}[cf.~with~\cite{KS-Ma}]
  A Hamiltonian operator $A\in\CDiffskad(\hat\kappa,\kappa)$ and a Nijenhuis
  operator $R\in\CDiff(\kappa,\kappa)$ constitute a \emph{variational
    Poisson--Nijenhuis structure} $(A,R)$ on $J^\infty(\pi)$ if the following
  compatibility conditions hold
\begin{align}
  \text{(i)}&\quad R\circ A=A\circ R^*,\\
  \text{(ii)}&\quad C(A,R)(\psi_1,\psi_2)=L_{A\psi_1}(R^*\psi_2)-
  L_{A\psi_2}(R^*\psi_1)+ R^*
  L_{A\psi_2}(\psi_1)-\nonumber\\
  &\quad R^* L_{A\psi_1}(\psi_2)+ \mathbf{E}\langle\psi_1,AR\psi_2\rangle-
  R^*\mathbf{E}\langle\psi_1,A\psi_2\rangle=0,
\end{align}
where $\mathbf{E}\colon\bar H^n(\pi)\rightarrow\hat\kappa$ is the Euler
operator and $\bar H^n(\pi)$ is the $n$th horizontal de~Rham cohomology group,
while $\langle.\,,.\rangle\colon\hat\kappa\times\kappa\to\bar H^n(\pi)$ is the
natural pairing.
\end{defn}
In terms of linearization operators condition~(ii) has the form
\begin{multline*}
  C(A,R)(\psi_1,\psi_2)=-\ell_{R^*,\,\psi_1}(A\psi_2)+
  \ell_{R^*,\,\psi_2}(A\psi_1)+
  \ell^*_{A,\,\psi_1}(R^*\psi_2)\\
  +\ell^*_{R^*,\,\psi_1}(A\psi_2)-R^*(\ell^*_{A,\,\psi_1}(\psi_2))=0.
\end{multline*}
Similarly to the finite-dimensional case, we have the following
\begin{prop}
  Let a Hamiltonian operator $A\in\CDiffskad(\hat\kappa,\kappa)$ and a
  Nijenhuis operator $R\in\CDiff(\kappa,\kappa)$ define a Poisson--Nijenhuis
  structure on $J^\infty(\pi)$. Then the composition $R\circ A$ is a
  Hamiltonian operator compatible with $A$.
\end{prop}

\begin{proof}
  By straightforward computations one can prove that
  \begin{multline}\label{eq:comp-cond-1}
    \ldb RA,RA\rdb(\psi_1,\psi_2)-
    2R\ldb A,RA\rdb(\psi_1,\psi_2)\\
    +R^2\ldb A,A\rdb(\psi_1,\psi_2)- [R,R]_{FN}(A\psi_1,A\psi_2)=0
  \end{multline}
  and
  \begin{multline}\label{eq:comp-cond-2}
    2\ldb A,RA\rdb(\psi_1,\psi_2)-
    \ldb A,A\rdb(R^*\psi_1,\psi_2)\\
    -\ldb A,A\rdb(\psi_1,R^*\psi_2)-
    2A\left(C(A,R)(\psi_1,\psi_2)\right)=0,
  \end{multline}
  from where the statement follows immediately.
\end{proof}

\begin{theor}\label{theor:main}
  Let a Hamiltonian operator $A\in\CDiffskad(\hat\kappa,\kappa)$ and a
  Nijenhuis operator $R\in\CDiff(\kappa,\kappa)$ define Poisson--Nijenhuis
  structure on $J^\infty(\pi)$. Then on $J^\infty(\pi)$ there is a hierarchy
  of iterated Hamiltonian operators\textup{,} that is a sequence of
  Hamiltonian operators $R^iA$\textup{,} $i\geq0$\textup{,} which are
  pair-wise compatible\textup{,} i.e.\textup{,} $\ldb
  R^iA,R^jA\rdb=0$\textup{,} $i,j\geq0$.
\end{theor}

\begin{proof}
  The proof is by induction on $n=\max(i,j)$. For $n=1$ the statement follows
  from the proposition above. Assume now that $\ldb R^iA,R^jA\rdb=0$,
  $i,j=0,\dots,n$, and $A(C(R^iA,R^j))=0$, $i+j\leq n$, and let us prove that
  $\ldb R^iA,R^{n+1}A\rdb=0$, $i=0,\dots,n+1$, and $A(C(R^iA,R^j))=0$,
  $i+j\leq n+1$.

  First, note that by direct computations one can prove the following formulas
  \begin{align}
    &\ldb RA,RB\rdb(\psi_1,\psi_2)- R\ldb RA,B\rdb(\psi_1,\psi_2)-
    R\ldb A,RB\rdb(\psi_1,\psi_2)+R^2\ldb A,B\rdb(\psi_1,\psi_2)\nonumber\\
    &\label{eq:comp-cond-1a}\qquad- [R,R]_{FN}(A\psi_1,B\psi_2)-
    [R,R]_{FN}(B\psi_1,A\psi_2)=0,\\
    &\ldb RA,B\rdb(\psi_1,\psi_2)+ \ldb A,RB\rdb(\psi_1,\psi_2)-
    \ldb A,B\rdb(R^*\psi_1,\psi_2)-\ldb A,B\rdb(\psi_1,R^*\psi_2)\nonumber\\
    &\label{eq:comp-cond-2a}\qquad- A\left(C(B,R)(\psi_1,\psi_2)\right)-
    B\left(C(A,R)(\psi_1,\psi_2)\right)=0,\\
    &C(RA,R)(\psi_1,\psi_2)+
    C(A,R^2)(\psi_1,\psi_2)-C(A,R)(R^*\psi_1,\psi_2)\nonumber\\
    &\label{eq:comp-cond-3}\qquad- C(A,R)(\psi_1,R^*\psi_2)-
    R^*\left(C(A,R)(\psi_1,\psi_2)\right)=0.
  \end{align}
  Let us substitute $R^nA$ and $R^{n-l}A$, $l=1,\dots,n$, for $A$ and $B$
  in~\eqref{eq:comp-cond-2a}, respectively. Then we get
  \begin{equation}\label{eq2-proof2}
    \ldb R^{n+1}A,R^{n-l}A\rdb-R^{n-l}A\left(C(R^{n}A,R)\right)=0.
  \end{equation}
  Now let us take $R^{n-1}A$, $R^{n-l}A$, $l=1,\dots,n$, and $R^2$ for $A$, $B$
  and $R$ in~\eqref{eq:comp-cond-2a}, respectively. Then we have
  \begin{equation}\label{eq2-proof3}
    \ldb R^{n+1}A,R^{n-l}A\rdb-R^{n-l}A\left(C(R^{n-1}A,R^2)\right)=0.
  \end{equation}
  If we have in~\eqref{eq:comp-cond-3} $R^{n-1}A$ for $A$ we get
  \begin{equation}\label{eq6-proof4}
    C(R^nA,R)+C(R^{n-1}A,R^2)=0.
  \end{equation}
  Therefore, taking the sum of~\eqref{eq2-proof2} and~\eqref{eq2-proof3} we
  obtain that $\ldb R^{n+1}A,R^{n-l}A\rdb=0$ for $l=1,\dots,n$.

  Let us substitute now $R^{n-1}A$ and $R^nA$ for $A$ and $B$
  in~\eqref{eq:comp-cond-1a}, respectively, and then $R^nA$ for $A$
  in~\eqref{eq:comp-cond-1}. Thus we get $\ldb R^{n+1}A,R^{n}A\rdb=0$ and
  $\ldb R^{n+1}A,R^{n+1}A\rdb=0$. In order to prove that $A(C(R^iA,R^j))=0$
  for $i+j\leq n+1$, one has to put $B=R^{n-l}A$, $l=0,\dots,n$ and take
  $R^{l+1}$ for $R$ in~\eqref{eq:comp-cond-2a}.
\end{proof}

For subsequent constructions we need the operator
\begin{multline*}
  C^*(A,R)(\psi,\phi)=-\ell_{A,\,\psi}(R\phi)+ \ell_{R,\,\phi}(A\psi)+
  R(\ell_{A,\,\psi}(\phi))\\
  +A(\ell^*_{R,\,\phi}(\psi)-\ell_{R^*,\,\psi}(\phi))
\end{multline*}
defined by
\begin{equation}\label{eq:15}
  \langle C(A,R)(\psi_1,\psi_2),\phi\rangle=
  \langle\psi_2,C^*(A,R)(\psi_1,\phi)\rangle.
\end{equation}

% ------------------------------------------------------

\section{Variational Poisson--Nijenhuis structures on evolution equations}\label{sec:vari-poiss-nijenh}
\subsection{Symmetries and cosymmetries}

Consider a system of evolution equations
\begin{equation}\label{eq:ev}
  \mathcal{E}=\{F=u_t-f(x,t,u,u_1,\dots,u_k)=0\},
\end{equation}
where both $u=(u^1,\dots,u^m)$ and $f=(f^1,\dots,f^m)$ are vectors and
$u_t=\partial u/\partial t$, $u_k=\partial^k u/\partial x_k$. For simplicity, we consider the case
of one space variable $x$, though everything works in general situation as
well. Denote by~$\mathcal{F}(\mathcal{E})$ the algebra of smooth functions
on~$\mathcal{E}$.

Recall that equation~\eqref{eq:ev} $\mathcal{E}$ can be understood as the space
$J^\infty(\pi)\times\mathbb{R}$ with the Cartan distribution generated by the
fields~$D_x$ and $D_t=\partial/\partial t+\re_f$. Here~$t$ the coordinate along
$\mathbb{R}$. In local coordinates, these fields are of the form
\begin{equation*}
  D_x=\frac{\partial}{\partial x}+\sum_{j=1}^m\sum_{k\geq 0}u^j_{k+1}\frac{\partial}{\partial
    u^j_k},\quad
  D_t=\frac{\partial}{\partial t}+\sum_{j=1}^m\sum_{k\geq 0}D^k_x(f^j)\frac{\partial}{\partial u^j_k}.
\end{equation*}

A \emph{symmetry} of the equation~$\mathcal{E}$ is a $\pi_\infty$-vertical
vector field on~$\mathcal{E}$ that preserves the Cartan distribution. The set
of all symmetries forms a Lie algebra over $\mathbb{R}$ denoted
by~$\sym(\mathcal{E})$ and there is a one-to-one correspondence
between~$\sym(\mathcal{E})$ and smooth
sections~$\phi\in\Gamma(\pi_\infty^*(\pi))=\kappa(\mathcal{E})$ satisfying the
equation
\begin{equation}  \label{eq:sym}
  \ell_{\mathcal{E}}(\phi)=0,
\end{equation}
where~$\ell_{\mathcal{E}}=D_t-\ell_f$ is the linearization operator
of~$\mathcal{E}$.

A \emph{conservation law} for the equation \eqref{eq:ev} is a horizontal
$1$-form~$\eta=Xdx+Tdt$ closed with respect to horizontal de Rham differential
$\bar d\colon\bar\Lambda^1(\mathcal{E})\rightarrow\bar\Lambda^2(\mathcal{E})$, i.e., such that
\begin{equation*}
  D_t(X)=D_x(T),
\end{equation*}
where $X,T\in\mathcal{F}(\mathcal{E})$. A conservation law $\eta$ is trivial if it is of the
form $\eta=\bar d h$, $h\in\mathcal{F}(\mathcal{E})$. The space of equivalence classes of
conservation laws coincides with the first horizontal de~Rham cohomology group
and is denoted by $\bar H^1(\mathcal{E})$.  To any conservation law $\eta=Xdx+Tdt$
there correspond its generating function $\psi_\eta=\mathbf{E}(\eta)$ that
satisfies the equation
\begin{equation}\label{eq:gf}
 \ell^*_\mathcal{E}(\psi_\eta)=0.
\end{equation}
Solutions of the last equation are called \emph{cosymmetries} of
equation $\mathcal{E}$ and the space of cosymmetries of $\mathcal{E}$ will be
denoted by $\sym^*(\mathcal{E})$.

% -----------------------------------------------------------------------

\subsection{Invariant Poisson--Nijenhuis structures}
\label{sec:invar-poiss-nijenh}

Consider the modules~$\kappa$ and~$\hat\kappa$ on the
space~$J^\infty(\pi)\times\mathbb{R}$ of extended jets, i.e., we admit
explicit dependence of their elements on~$t$. Since vector fields
on~$\mathcal{E}$ act on~$\kappa$ and~$\hat\kappa$ by Lie derivatives, we can
give the following
\begin{defn}
  An operator~$O$ acting from~$\kappa$ to~$\kappa$ (or from~$\kappa$
  to~$\hat\kappa$, etc.) is called invariant if~$L_{D_t}\circ O=O\circ
  L_{D_t}$.
\end{defn}

\begin{prop}
  An operator~$A\colon\hat\kappa\to\kappa$ is invariant iff
  \begin{equation}
    \label{eq:2}
    \ell_F\circ A+A\circ\ell_F^*=0.
  \end{equation}
  An operator~$R\colon\kappa\to\kappa$ is invariant iff
  \begin{equation}
    \label{eq:4}
    \ell_F\circ R-R\circ\ell_F=0.
  \end{equation}
\end{prop}

\begin{proof}
  Indeed, from~\eqref{eq:1} one has for the action on~$\kappa$
  \begin{equation*}
    L_{D_t}=L_{\pd{}{t}+\re_f}=\pd{}{t}+\re_f-\ell_f=\ell_F,
  \end{equation*}
  while from~\eqref{eq:3} for the action on~$\hat\kappa$ we obtain
  \begin{equation*}
    L_{D_t}=L_{\pd{}{t}+\re_f}=\pd{}{t}+\re_f+\ell_f^*=-\ell_F^*
  \end{equation*}
  and this proves both~\eqref{eq:2} and~\eqref{eq:4}.
\end{proof}

\begin{rem}
  From~\eqref{eq:4} we see that an invariant $\mathcal{C}$-differential
  operatos takes symmetries of equation~$\mathcal{E}$ to symmetries, i.e., is a
  \emph{recursion operator} for symmetries of~$\mathcal{E}$. On the other
  hand, a $\mathcal{C}$-differential operator~$A$ that enjoys~\eqref{eq:2}
  takes cosymmetries to symmetries. As it was shown in~\cite{Ham}, if~$A$ is
  skew-adjoint and satisfies~$\ldb A,A\rdb=0$ it is a \emph{Hamiltonian
    structure} for~$\mathcal{E}$ and all Hamiltonian structures can be found
  in such a way.
\end{rem}

\begin{defn}
  Let~$(A,R)$ be a Poisson--Nijenhuis structure
  on~$J^\infty(\pi)\times\mathbb{R}$. We say that it is a Poisson--Nijenhuis
  structure on~$\mathcal{E}$ if both~$A$ and~$R$ are invariant operators.
\end{defn}

Thus, a Poisson--Nijenhuis on~$\mathcal{E}$ consists of a Hamiltonian
structure~$A$ and a recursion operator which is a Nijenhuis operator
compatible with~$A$.

%---------------------------------------------------------------------
\subsection{The $\ell^*$- and $\ell$-coverings.}
\label{sec:ell-ell-coverings}

We shall now look at Poisson--Nijenhuis structures from a different point of
view. To this end, consider the following extension of the equation~$\mathcal{E}$. Let
us add to $\mathcal{E}$ new \emph{odd} variable~$p=(p^1,\dots,p^m)$ that satisfies
\begin{equation}
  \label{eq:5}
  p_t=-\ell_f^*(p).
\end{equation}
The system consisting of the initial equation~$\mathcal{E}$ and
equation~\eqref{eq:5} is called the \emph{$\ell^*$-covering} of~$\mathcal{E}$
and is denoted by~ $\mathcal{L}^*\mathcal{E}$.  The extended total derivatives on
$\mathcal{L}^*\mathcal{E}$ are
\begin{equation*}
  \tilde D_x= D_x+\sum_{j=1}^m\sum_{k\geq 0} p^j_{k+1}\frac{\partial}{\partial p^j_k},\quad
  \tilde D_t=D_t-
  \sum_{j=1}^m\sum_{k\geq 0}
  \tilde D^k_x(\tilde\ell^*_f(p^j))\frac{\partial}{\partial p^j_k}.
\end{equation*}
Note that any $\mathcal{C}$-differential operator~$O$ on~$\mathcal{E}$ can be lifted to an
operator~$\tilde{O}$ on~$\mathcal{L}^*\mathcal{E}$ by superscribing tildes over
corresponding total derivatives.

Let~$A\colon\hat\kappa\to\kappa$ be a $\mathcal{C}$-differential operator of
the form~$\Vert\sum_{i\geq 0}a^l_{ij}D^i_x\Vert$. Consider a $p$-linear
vector function~$\mathcal{H}_A=(\mathcal{H}_A^1,\dots,\mathcal{H}_A^m)$,
\begin{equation}\label{ellast:solution}
 \mathcal{H}_A^{\,l}=\sum_{i,j}a^l_{ij}p^j_i, \qquad a^l_{ij}\in\mathcal{F}(\mathcal{E}).
\end{equation}

\begin{theor}[see~\cite{Ham}]\label{sec:ell-ell-coverings-4}
  A skew-adjoint operator~$A$ satisfies~\eqref{eq:2} iff
  \begin{equation}\label{eq:ellast}
    \tilde{\ell}_\mathcal{E}(\mathcal{H}_A)=0.
  \end{equation}
\end{theor}

In the terminology of the covering theory~\cite{Nonloc}, a vector function
satisfying~\eqref{eq:ellast} is nothing but a \emph{shadow} of symmetry of
$\mathcal{E}$ in the $\ell^*$-covering or, more precisely, $\mathcal{H}_A$ is
generating section of the shadow
\begin{equation*}
  \tilde\re_{\mathcal{H}_A}=\sum_{k\geq0}\tilde D_x^k(\mathcal{H}_A^j)\pd{}{u_k^j}.
\end{equation*}
In coordinate-free terms, a shadow is a derivation
of~$\mathcal{F}(\mathcal{E})$ with values in the function algebra
on~$\mathcal{L}^*\mathcal{E}$ that preserves the Cartan distribution.

\begin{lemma}
  \label{sec:ell-ell-coverings-1}
  Let~$A\in\CDiffskad(\hat\kappa,\kappa)$ and~$\tilde{\re}_{\mathcal{H}_A}$ be the
  corresponding shadow. Then there exists a symmetry of the
  equation~$\mathcal{L}^*\mathcal{E}$ such that its restriction
  to~$\mathcal{F}(\mathcal{E})$ coincides with~$\tilde{\re}_{\mathcal{H}_A}$.
\end{lemma}

\begin{proof}
  Consider the vector function~$\alpha=(\alpha^1,\dots,\alpha^m)$ defined by
  \begin{equation*}
    \alpha=-\frac{1}{2}\tilde{\ell}_{\mathcal{H}_A}^*(p)
  \end{equation*}
  and set
  \begin{equation}
    \label{eq:6}
    \tilde{\re}_{\mathcal{H}_A,\bar{\mathcal{H}}_A}=\tilde{\re}_{\mathcal{H}_A}+
    \sum_{k,j}\tilde{D}_x^k(\alpha^j)\pd{}{p_k^j}.
  \end{equation}
  It is easily checked that the
  field~$\tilde{\re}_{\mathcal{H}_A,\bar{\mathcal{H}}_A}$ is the desired symmetry.
\end{proof}

Consider now two operators~$A$, $B\in\CDiffskad(\hat\kappa,\kappa)$ and using
Lemma~\ref{sec:ell-ell-coverings-1} define the \emph{Jacobi bracket}
of~$\mathcal{H}_A$ and~$\mathcal{H}_B$ by
\begin{equation}
  \label{eq:7}
  \{\mathcal{H}_A,\mathcal{H}_B\}=\tilde{\re}_{\mathcal{H}_A,\bar{\mathcal{H}}_A}(\mathcal{H}_B)
  +\tilde{\re}_{\mathcal{H}_B,\bar{\mathcal{H}}_B}(\mathcal{H}_A)
\end{equation}

\begin{rem}
  The plus sign in the right-hand side of equation~\eqref{eq:7} is due to the
  fact that both the fields~$\tilde{\re}_{\mathcal{H}_A,\bar{\mathcal{H}_A}}$,
  $\tilde{\re}_{\mathcal{H}_B,\bar{\mathcal{H}_B}}$ and the
  functions~$\mathcal{H}_B$, $\mathcal{H}_A$ are odd.
\end{rem}

\begin{rem}
  In more explicit terms the Jacobi bracket can be rewritten in the form
  \begin{equation}\label{eq:16}
    \{\mathcal{H}_A,\mathcal{H}_B\}=
    -\tilde{\ell}_{\mathcal{H}_A}(\mathcal{H}_B)-
    \tilde{\ell}_{\mathcal{H}_B}(\mathcal{H}_A)-
    \frac{A(\tilde{\ell}^*_{\mathcal{H}_B}(p))+
      B(\tilde{\ell}^*_{\mathcal{H}_A}(p))}{2}.
\end{equation}
\end{rem}

Note that there exists a one-to-one correspondence between
$\mathcal{C}$-differential operators from the module
$\CDiffskew_{(k)}(\hat\kappa,\kappa)$ and functions
on~$\mathcal{L}^*\mathcal{E}$ that are $(k-1)$-linear with respect to the odd
variables. For~$\Delta\in\CDiffskew_{(k)}(\hat\kappa,\kappa)$ denote
by~$\mathcal{H}_\Delta$ the corresponding function.  Then by direct
computations one can prove the following

\begin{theor}\label{sec:ell-ell-coverings-2}
  Let $\mathcal{H}_A$ and $\mathcal{H}_B$ be shadows of symmetries in the
  $\ell^*$-covering to which there correspond operators
  $A,B\in\CDiffskad(\hat\kappa,\kappa)$ . Then
  \begin{equation}\label{eq:8}
    \{\mathcal{H}_A,\mathcal{H}_B\}=\mathcal{H}_{\ldb A,B\rdb}.
  \end{equation}
\end{theor}

If $\psi_1$ and $\psi_2$ are sections of the $\ell^*_\mathcal{E}$-covering
such that $(d\psi_i)(\mathcal{C})\subset\tilde{\mathcal{C}}$, $i=1,2$, where
$\mathcal{C}$ and $\tilde{\mathcal{C}}$ are the Cartan distributions on
$\mathcal{E}$ and $\mathcal{L}^*\mathcal{E}$, respectively, then
equation~\eqref{eq:8} can be rewritten as follows
\begin{equation*}
  \psi_1^*\psi_2^*\{\mathcal{H}_A,\mathcal{H}_B\}=\ldb A,B\rdb(\psi_1,\psi_2).
\end{equation*}

We now introduce the object dual to the $\ell^*$-covering. Namely, consider
the extension of~$\mathcal{E}$ by new \emph{odd} variables~$q=(q^1,\dots,q^m)$
that satisfy the equation
\begin{equation}\label{eq:9}
  q_t=\ell_f(q).
\end{equation}
The system consisting of the initial equation and equation~\eqref{eq:9} will
be called the \emph{$\ell$-covering} of~$\mathcal{E}$ and denoted
by~$\mathcal{L}\mathcal{E}$. The extended total derivatives on
$\mathcal{L}\mathcal{E}$ are
\begin{equation*}
  \tilde D_x= D_x+\sum_{l=1}^m\sum_{k\geq 0}q^l_{k+1}
  \frac{\partial}{\partial q^l_k},\quad
  \tilde D_t=D_t+\sum_{l=1}^m\sum_{k\geq 0}
  \tilde D^k_x(\tilde\ell_f(q^l))\frac{\partial}{\partial q^l_k}.
\end{equation*}

To any $\mathcal{C}$-differential operator~$R\colon\kappa\to\kappa$ of the
form~$\sum_{i\ge0}a_{ij}^lD_x^i$ let us put into correspondence a $q$-linear
vector function~$\mathcal{N}_R=(\mathcal{N}_R^1,\dots,\mathcal{N}_R^m)$, where
\begin{equation*}
  \mathcal{N}_R^{\,l}=\sum_{i,j}a_{ij}^lq_i^j, \quad
 a_{ij}^l\in\mathcal{F}(\mathcal{E}).
\end{equation*}
Similar to Theorem~\ref{sec:ell-ell-coverings-4} we have the following

\begin{theor}[see~\cite{Ham,Bussinesq}]
  \label{sec:ell-ell-coverings-3}
  An operator~$R$ satisfies relation~\eqref{eq:4}\textup{,} i.e\textup{,} is a
  recursion operator for the equation~$\mathcal{E}$ iff
  \begin{equation}
    \label{eq:10}
    \tilde{\ell}_{\mathcal{E}}(\mathcal{N}_R)=0,
  \end{equation}
  i.e.\textup{,}~$\mathcal{N}_R$ is a shadow of a symmetry in the
  $\ell$-covering.
\end{theor}

The corresponding derivation is of the form
\begin{equation*}
  \tilde{\re}_{\mathcal{N}_R}=\sum_{k\geq0}\tilde D_x^k(\mathcal{N}_R^j)\pd{}{u_k^j}.
\end{equation*}
In parallel to Lemma~\ref{sec:ell-ell-coverings-1} we have the following
auxiliary result.

\begin{lemma}
  \label{sec:ell-ell-coverings-5}
  Let~$R\in\CDiff(\kappa,\kappa)$ and~$\tilde{\re}_{\mathcal{N}_R}$ be the
  corresponding shadow. Then there exists a symmetry of the
  equation~$\mathcal{L}\mathcal{E}$ such that its restriction
  to~$\mathcal{F}(\mathcal{E})$ coincides with~$\tilde{\re}_{\mathcal{N}_R}$.
\end{lemma}

\begin{proof}
  Consider the vector function~$\beta=(\beta^1,\dots,\beta^m)$ with
  \begin{equation*}
    \beta^j=\sum_{k,l}\pd{\mathcal{N}_R^j}{u_k^l}q_k^l
  \end{equation*}
  and set
  \begin{equation}
    \label{eq:11}
    \tilde{\re}_{\mathcal{N}_R,\bar{\mathcal{N}}_R}=\tilde{\re}_{\mathcal{N}_R}+
    \sum_{k,j}\tilde{D}_x^k(\beta^j)\pd{}{q_k^j}.
  \end{equation}
  This is the symmetry we are looking for.
\end{proof}

Using Lemma~\ref{sec:ell-ell-coverings-5} we define the Jacobi bracket of
vector functions~$\mathcal{N}_R$ and~$\mathcal{N}_S$ by
\begin{equation}
  \label{eq:12}
  \{\mathcal{N}_R,\mathcal{N}_S\}=\tilde{\re}_{\mathcal{N}_R,\bar{\mathcal{N}}_R}(\mathcal{N}_S)
  +\tilde{\re}_{\mathcal{N}_S,\bar{\mathcal{N}}_S}(\mathcal{N}_R).
\end{equation}
In explicit terms the Jacobi bracket has the form
\begin{equation}\label{eq:17}
  \{\mathcal{N}_R,\mathcal{N}_S\}=-\tilde{\ell}_{\mathcal{N}_R}(\mathcal{N}_S)
  -\tilde{\ell}_{\mathcal{N}_S}(\mathcal{N}_R)+R(\tilde{\ell}_{\mathcal{N}_S}(q))
  +S(\tilde{\ell}_{\mathcal{N}_R}(q)).
\end{equation}

Denote by $\mathcal{N}_{[R,S]_{FN}}$ the function on $\mathcal{L}\mathcal{E}$
bilinear with respect to the variables~$q$ and corresponding to the
operator~$[R,S]_{FN}\in\CDiffskew_{(2)}(\kappa,\kappa)$. Then by the direct
computations we obtain

\begin{theor}
  Let $\mathcal{N}_R$ and~$\mathcal{N}_S$ be shadows of symmetries in the
  $\ell$-covering to which there correspond operators~$R$
  and~$S\in\CDiff(\kappa,\kappa)$. Then
  \begin{equation*}
    \{\mathcal{N}_R,\mathcal{N}_S\}=\mathcal{N}_{[R,S]_{FN}}.
  \end{equation*}
\end{theor}

If~$\phi_1$, $\phi_2$ are sections of the $\ell$-covering preserving the
Cartan distributions then
\begin{equation*}
  \phi_1^*\phi_2^*\{\mathcal{N}_R,\mathcal{N}_S\}=[R,S]_{FN}(\phi_1,\phi_2).
\end{equation*}

%---------------------------------------------------------------------------------------------
\subsection{Compatibility condition} \label{sec:compatibility}

We shall now express the compatibility condition of a Hamiltonian
structure~$A$ and a recursion operator~$R$ on equation~\eqref{eq:ev} in
similar geometric terms. To this end, recall that
both~$\mathcal{L}\mathcal{E}$ and~$\mathcal{L}^*\mathcal{E}$ are fibered over
the equation~$\mathcal{E}$ and denote the corresponding fiber bundles
by~$\tau\colon\mathcal{L}\mathcal{E}\to\mathcal{E}$
and~$\tau^*\colon\mathcal{L}\mathcal{E}\to\mathcal{E}$, respectively. Consider
the Whitney product $\tau\otimes\tau^*\colon
\mathcal{L}\mathcal{E}\times_{\mathcal{E}}\mathcal{L}^*\mathcal{E}
\to\mathcal{E}$ of these bundles. One can extend the total derivatives to the
space~$\mathcal{L}\mathcal{E}\times_{\mathcal{E}}\mathcal{L}^*\mathcal{E}$ by
setting
\begin{align*}
  \tilde D_x&= D_x+\sum_{l=1}^m\sum_{k\geq 0}\left(q^l_{k+1}\frac{\partial}{\partial q^l_k}
    +p^{\,l}_{k+1}\frac{\partial}{\partial p^{\,l}_k}\right),\\
  \tilde D_t&=D_t+\sum_{l=1}^m\sum_{k\geq 0}\left(\tilde D^k_x(\tilde\ell_f(q^l))\frac{\partial}{\partial q^l_k}
 -\tilde D^k_x(\tilde\ell^*_f(p^{\,l}))\frac{\partial}{\partial p^{\,l}_k}\right).
\end{align*}
Thus the
equation~$\mathcal{L}\mathcal{E}\times_{\mathcal{E}}\mathcal{L}^*\mathcal{E}$
amounts to~$\mathcal{E}$ extended both by~\eqref{eq:5} and~\eqref{eq:9}.

Due to the natural
projections~$\mathcal{L}\mathcal{E}\times_{\mathcal{E}}\mathcal{L}^*\mathcal{E}\to\mathcal{L}^*\mathcal{E}$
and~$\mathcal{L}\mathcal{E}\times_{\mathcal{E}}\mathcal{L}^*\mathcal{E}\to\mathcal{L}\mathcal{E}$,
the vector fields~$\tilde{\re}_{\mathcal{H}_A,\bar{\mathcal{H}}_A}$
and~$\tilde{\re}_{\mathcal{N}_R,\bar{\mathcal{N}}_R}$ (see
equalities~\eqref{eq:6} and~\eqref{eq:11}) may be considered as derivations
from~$\mathcal{F}(\mathcal{L}^*\mathcal{E})$
and~$\mathcal{F}(\mathcal{L}\mathcal{E})$
to~$\mathcal{F}(\mathcal{L}\mathcal{E}\times_{\mathcal{E}}\mathcal{L}^*\mathcal{E})$,
respectively.

\begin{lemma}
  \label{sec:comp-cond}
  There exist
  symmetries~$\tilde{\re}_{\mathcal{H}_A,\bar{\mathcal{H}}_A,\alpha}$
  and~$\tilde{\re}_{\mathcal{N}_R,\bar{\mathcal{N}}_R,\rho}$ of
  equation~$\mathcal{L}\mathcal{E}\times_{\mathcal{E}}\mathcal{L}^*\mathcal{E}$
  such that their restrictions to the function
  algebras~$\mathcal{F}(\mathcal{L}^*\mathcal{E})$
  and~$\mathcal{F}(\mathcal{L}\mathcal{E})$ coincide
  with~$\tilde{\re}_{\mathcal{H}_A,\bar{\mathcal{H}}_A}$
  and~$\tilde{\re}_{\mathcal{N}_R,\bar{\mathcal{N}}_R}$\textup{,} respectively.
\end{lemma}

\begin{proof}
  Let us set
  \begin{equation}
    \label{eq:13}
    \tilde{\re}_{\mathcal{H}_A,\bar{\mathcal{H}}_A,\alpha}=
    \tilde{\re}_{\mathcal{H}_A,\bar{\mathcal{H}}_A}+\sum_{k,j}\tilde{D}_x^k(\alpha^j)\pd{}{q_k^j},
  \end{equation}
  where the vector function~$\alpha=(\alpha^1,\dots,\alpha^m)$ is defined
  by~$\alpha=\tilde{\ell}_{\mathcal{H}_A}(q)$. Let also consider the field
  \begin{equation}
    \label{eq:14}
    \tilde{\re}_{\mathcal{N}_R,\bar{\mathcal{N}}_R,\rho}=
    \tilde{\re}_{\mathcal{N}_R,\bar{\mathcal{N}}_R}+\sum_{k,j}\tilde{D}_x^k(\rho^j)\pd{}{p_k^j},
  \end{equation}
  where~$\rho=(\rho^1,\dots,\rho^m)$ is defined by
  \begin{equation*}
    \rho=-\tilde{\ell}^*_{\mathcal{N}_R}(p)-\tilde{\re}_q(R^*)(p).
  \end{equation*}
  By direct computations one can check that the fields~\eqref{eq:13}
  and~\eqref{eq:14} possess the needed properties.
\end{proof}

Using Lemma~\ref{sec:comp-cond} let us define the bracket
\begin{equation*}
  \{\mathcal{H}_A,\mathcal{N}_R\}=\tilde\re_{\mathcal{H}_A,\bar{\mathcal{H}}_A,\alpha}(\mathcal{N}_R)+
  \tilde\re_{\mathcal{N}_R,\bar{\mathcal{N}}_R,\rho}(\mathcal{H}_A),
\end{equation*}
or, in explicit terms,
\begin{equation}\label{eq:18}
  \{\mathcal{H}_A,\mathcal{N}_R\}=
  -\tilde{\ell}_{\mathcal{N}_R}(\mathcal{H}_A)-\tilde{\ell}_{\mathcal{H}_A}(\mathcal{N}_R)-
  A\left(\tilde{\ell}^*_{\mathcal{N}_R}(p)+\tilde{\re}_{q}(R^*)(p)\right)+
  R(\tilde{\ell}_{\mathcal{H}_A}(q)).
\end{equation}

Denote by $\mathcal{C}_{A,R}^*$ the bilinear with respect to the variables~$p$
and~$q$ function on
$\mathcal{L}\mathcal{E}\times_{\mathcal{E}}\mathcal{L}^*\mathcal{E}$
corresponding to the $\mathcal{C}$-differential operator
$C^*(A,R)\colon\hat\kappa\times\kappa\rightarrow\kappa$ (see
equality~\eqref{eq:15}). Then the following result holds:

\begin{theor}
  Let $\mathcal{H}_A$ be a shadow in the $\ell^*$-covering to which there
  corresponds an operator $A\in\CDiffskad(\hat\kappa,\kappa)$ and
  $\mathcal{N}_R$ be a shadow in the $\ell$-covering to which there
  corresponds an operator $R\in\CDiff(\kappa,\kappa)$. Then
  \begin{equation*}
    \{\mathcal{H}_A,\mathcal{N}_R\}=\mathcal{C}_{A,R}^*.
  \end{equation*}
\end{theor}
From the above said we get the following

\begin{theor}
  Let a Hamiltonian operator $A\in\CDiffskad(\hat\kappa,\kappa)$ and a
  recursion operator $R\in\CDiff(\kappa,\kappa)$ define Poisson--Nijenhuis
  structure on evolution equation~$\mathcal{E}$\textup{,} while $\mathcal{H}_A$ and
  $\mathcal{N}_R$ be the corresponding shadows in the $\ell^*$- and
  $\ell$-covering over $\mathcal{E}$\textup{,} respectively. Then
  \begin{align}
    \textup{(i)}& \quad \{\mathcal{H}_A,\mathcal{H}_A\}=0,\\
    \textup{(ii)}& \quad \{\mathcal{N}_R,\mathcal{N}_R\}=0,\\
    \textup{(iii)}& \quad \{\mathcal{H}_A,\mathcal{N}_R\}=0.
  \end{align}
\end{theor}

% ------------------------------------------------------------------------

\section{Variational Poisson--Nijenhuis structures in general case}
\label{sec:gen-str}

Consider now the infinite prolongation~$\mathcal{E}\subset J^\infty(\pi)$ of a
general differential equation as a submanifold in the space~$J^\infty(\pi)$ of
infinite jets of some locally trivial bundle~$\pi\colon E\to M$. In local
coordinates~$\mathcal{E}$ is given by the system
\begin{equation*}
  F^j(x^1,\dots,x^n,u^1,\dots,u^m,\dots,u^j_\sigma,\dots)=0,\qquad
  j=1,\dots,r.
\end{equation*}
We assume that~$F=(F^1,\dots,F^r)\in P$, where~$P$ is the module of sections
of some vector bundle over~$J^\infty(\pi)$. Consider the linearization
operator~$\ell_{\mathcal{E}}\colon\kappa\to P$ and its
adjoint~$\ell_{\mathcal{E}}^*\colon\hat{P}\to\hat{\kappa}$. Similar to the
evolutionary case, we can construct~$\ell$- and $\ell^*$-coverings by
extending~$\mathcal{E}$ with~$\ell_{\mathcal{E}}(q)=0$
and~$\ell_{\mathcal{E}}^*(p)=0$.

Following the scheme exposed in Section~\ref{sec:vari-poiss-nijenh}, we are
looking for $\mathcal{C}$-differential operators such that the diagrams
\begin{equation}
  \label{eq:19}
  \textup{(i)}\quad
  \xymatrix{
    \kappa\ar[r]^{\ell_{\mathcal{E}}}\ar[d]_{R}&P\ar[d]^{\bar{R}}\\
    \kappa\ar[r]^{\ell_{\mathcal{E}}}&P\rlap{,}}\qquad\qquad\textup{(ii)}\quad
  \xymatrix{
    \hat{P}\ar[r]^{\ell_{\mathcal{E}}^*}\ar[d]_{A}&\hat{\kappa}\ar[d]^{\bar{A}}\\
    \kappa\ar[r]^{\ell_{\mathcal{E}}}&P}
\end{equation}
are commutative. Literally copying the construction of
Section~\ref{sec:vari-poiss-nijenh}, we can put into correspondence to an
operator~$R$ from the first diagram in~\eqref{eq:19} a $q$-linear
function~$\mathcal{N}_R=(\mathcal{N}_R^1,\dots,\mathcal{N}_R^m)$
on~$\mathcal{L}\mathcal{E}$, while for~$A$ from the second diagram we
construct a $p$-linear
function~$\mathcal{H}_A=(\mathcal{H}_A^1,\dots,\mathcal{H}_A^m)$
on~$\mathcal{L}^*\mathcal{E}$. We again use the general result proved
in~\cite{Ham}:

\begin{theor}
  An operator~$R$ fits diagram~\textup{(i)} in~\eqref{eq:19} iff
  \begin{equation}
    \label{eq:21}
    \tilde{\ell}_{\mathcal{E}}(\mathcal{N}_R)=0,
  \end{equation}
  while~$A$ fits diagram~\textup{(ii)} iff
  \begin{equation}\label{eq:22}
    \tilde{\ell}_{\mathcal{E}}(\mathcal{H}_A)=0.
  \end{equation}
\end{theor}

Among the operators~$A$ let us distinguish those ones that enjoy the property
similar to skew-adjointness. Namely, we shall consider operators such that
\begin{equation}
  \label{eq:20}
  A^*=-\bar{A}.
\end{equation}

\begin{rem}
  This property means that the
  operator~$(A,\bar{A})\colon\hat{P}\oplus\hat{\kappa}\to\kappa\oplus P$ is
  skew-adjoint.
\end{rem}

Note now that explicit expressions~\eqref{eq:16}, \eqref{eq:17}
and~\eqref{eq:18} for Jacobi brackets do not rely on the fact that they were
given for evolutionary equations. Using this observation we give the following

\begin{defn}
  \label{sec:vari-poiss-nijenh-1}
  Let~$\mathcal{E}\subset J^\infty(\pi)$ be a differential equation.
  \begin{enumerate}
  \item A $\mathcal{C}$-differential operator~$A\colon\hat{P}\to\kappa$ is
    called a \emph{Hamiltonian structure} on~$\mathcal{E}$ if it fits the left
    diagram~\eqref{eq:19}, equation~\eqref{eq:20} holds
    and~$\{\mathcal{H}_A,\mathcal{H}_A\}=0$, where the bracket is given
    by~\eqref{eq:16}. Two Hamiltonian structures~$A$ and~$B$ are said to be
    \emph{compatible} if~$\{\mathcal{H}_A,\mathcal{H}_B\}=0$.
  \item A $\mathcal{C}$-differential operator~$R\colon\kappa\to\kappa$ is
    called a \emph{Nijenhuis operator} for the equation~$\mathcal{E}$ if it
    fits the right diagram~\eqref{eq:19}
    and~$\{\mathcal{N}_R,\mathcal{N}_R\}=0$, where the bracket is given
    by~\eqref{eq:17}.
  \item A pair of $\mathcal{C}$-differential operators~$(A,R)$ is called a
    \emph{Poisson--Nijenhuis structure} on~$\mathcal{E}$ if~$R$ is a Nijenhuis
    operator, $A$ is a Hamiltonian structure such that~$A^*\circ
    R^*=\bar{R}\circ A^*$ and~$\{\mathcal{N}_R,\mathcal{H}_A\}=0$, where the
    bracket is given by~\eqref{eq:18}.
  \end{enumerate}
\end{defn}

\begin{rem}
  If~$(A,R)$ is a Poisson--Nijenhuis structure on~$\mathcal{E}$ then $R$ is a
  recursion operator for symmetries of~$\mathcal{E}$. Moreover, a Hamiltonian
  structure~$A$, similar to the evolutionary case, determines a Poisson
  bracket on the group of conservation laws of~$\mathcal{E}$. Namely,
  if~$\omega_1$, $\omega_2$ are conservation laws then we set
  \begin{equation*}
    \{\omega_1,\omega_2\}_A=L_{d_1^{0,n-1}\omega_1}(\omega_2),
  \end{equation*}
  where~$d_1^{0,n-1}\colon E_1^{0,n-1}\to E_1^{1,n-1}$ is the differential in
  Vinogradov's $\mathcal{C}$-spectral sequence, see~\cite{Vin-book}. For
  evolutionary equations this differential coincides with the Euler operator.
\end{rem}

The following result generalises Theorem~\ref{theor:main}:

\begin{theor}
  \label{sec:vari-poiss-nijenh-2}
  If~$(A,R)$ is a Poisson--Nijenhuis structure on~$\mathcal{E}$ then~$R^i\circ
  A$\textup{,} $i=0,1,2,\dots$\textup{,} is a family of pair-wise compatible
  Hamiltonian structures on~$\mathcal{E}$.
\end{theor}

% ------------------------------------------------------------------------

\section{Concluding remarks:  nonlocal Poisson--Nijenhuis structures}
\label{sec:nonloc-str}

Strictly speaking all constructions exposed above are valid for \emph{local}
Poisson--Nijenhuis structures. In reality, the operators~$A$ or~$R$ or both
are \emph{nonlocal}, i.e., contain terms like~$D_x^{-1}$. For example, recall
the recursion operator for the KdV equation. It seems that our approach can be
extended to structures of this type. The general scheme is as follows.

Let~$\mathcal{E}$ be a differential equation
and~$\tau\colon\widetilde{\mathcal{L}\mathcal{E}}\to\mathcal{L}\mathcal{E}$ be a
general covering over~$\mathcal{L}\mathcal{E}$ in the sense
of~\cite{Nonloc}. Then solutions of the equation
\begin{equation*}
  \tilde{\ell}_{\mathcal{E}}(\mathcal{N})=0
\end{equation*}
linear with respect to odd variables give rise to nonlocal
$\mathcal{C}$-differential recursion operators~$R_{\mathcal{N}}$ with
nonlocalities corresponding to nonlocal variables defined by~$\tau$. These
solutions are shadows of symmetries in this covering. The hardest problem lies
in definition of the Jacobi bracket for such shadows. Nontriviality of this
problem is illustrated by observation given in~\cite{Ghosts}. A way to commute
shadows can be derived from the results of~\cite{Khor'kova} but the
constructions given there are ambiguous.

In~\cite{Shad-Rec} we described a canonical construction for the Jacobi
bracket~$\{.\,,.\}$ of shadows of a general nature. Given this construction
and taking into account the above exposed results, we can define nonlocal
Nijenhuis operators~$R_{\mathcal{N}}$ as the ones satisfying
\begin{equation*}
  \{\mathcal{N},\mathcal{N}\}=0.
\end{equation*}

In a similar manner, we can consider coverings
over~$\tau^*\colon\widetilde{\mathcal{L}^*\mathcal{E}}\to\mathcal{L}^*\mathcal{E}$
and, solving the equation
\begin{equation*}
  \tilde{\ell}_{\mathcal{E}}(\mathcal{H})=0,
\end{equation*}
look for nonlocal Hamiltonian operators~$A_{\mathcal{H}}$ corresponding to
shadows~$\mathcal{H}$. The Hamiltonianity condition is given by
\begin{equation*}
  \{\mathcal{H},\mathcal{H}\}=0.
\end{equation*}

Finally, the compatibility condition for~$R_{\mathcal{N}}$ and~$A_\mathcal{H}$
are expressed by
\begin{equation*}
  \{\mathcal{H},\mathcal{N}\}=0,
\end{equation*}
where the bracket is considered in the Whitney product of~$\tau$ and~$\tau^*$.

A detailed theory of nonlocal Poisson--Nijenhuis structures will be given
elsewhere.

\begin{rem}
  As it was demonstrated in~\cite{Ham} and~\cite{Bussinesq}, a very efficient
  way to construct nonlocal recursion operators and Hamiltonian structures for
  evolution equations is the use of nonlocal vectors (for the
  $\ell^*$-covering) and covectors (for the $\ell$-covering). In particular,
  this method, by its nature, leads to the so-called \emph{weakly nonlocal
    operators}.

  When this paper was almost finished, Maria Clara Nucci indicated to us the
  work~\cite{Ibragim} where the construction of nonlocal vectors was
  reinvented. The author of~\cite{Ibragim} exploits the Lagrangian structure
  of the $\ell^*$-covering, though the reason for existence of nonlocal
  vectors is more general (it suffices to compare the construction with the
  one for nonlocal covectors on the $\ell$-covering).
\end{rem}

% ------------------------------------------------------------------------------

\end{document}